\documentclass{article}
\usepackage{amsmath,amssymb,amsfonts,dsfont,amsthm}
\usepackage{color}
\usepackage{enumerate}
\usepackage{graphicx,psfrag}
\usepackage{cite}
\graphicspath{{./Figures/}}

\providecommand{\com[2]}{\begin{tt}[#1: #2]\end{tt}}

\newcommand{\subscr}[2]{#1_{\textup{#2}}}
\newcommand{\supscr}[2]{#1^{\textup{#2}}}
\newcommand{\setdef}[2]{\{#1 \, : \; #2\}}

\providecommand{\quotes[1]}{``#1"}

\newcommand{\real}{\mathbb{R}}

\newcommand{\integernonnegative}{\mathbb{Z}_{\ge 0}}

\newtheorem{example}{Example}

\newtheorem{theorem}{Theorem}

\newtheorem{lemma}[theorem]{Lemma}

\newtheorem{proposition}[theorem]{Proposition}

\newcommand{\dst}{\displaystyle}
\newcommand{\be}{\begin{equation}}
\newcommand{\ee}{\end{equation}}
\renewcommand{\l}{\left}
\renewcommand{\r}{\right}

\providecommand{\prt}[1]{\left( #1 \right)} % parentheses

\newcommand{\Exp}{\mathbb{E}} %  Expectation
\newcommand{\1}{\mathbf{1}} % vector of ones

\newcommand{\barx}{\bar{x}}   % average state
\newcommand{\Amax}{\subscr{\supscr{a}{all}}{max}}	% max of sum of all aij
\newcommand{\armax}{\subscr{\supscr{a}{r}}{max}\,}	% max of sum on row
\newcommand{\acmax}{\subscr{\supscr{a}{c}}{max}\,}	% max of sum on column
\newcommand{\admin}{\subscr{\supscr{a}{d}}{min}\,}	% min on diagonal
\newcommand{\aindmax}{\subscr{\supscr{a}{ind}}{max}\,}	% max a_ij
\newcommand{\degmax}{\subscr{d\,}{\!max}} % maximum degree (only in the introduction)
\newcommand{\degmaxcol}{\supscr{\degmax}{\!col}} % maximum column degree

\providecommand{\abs[1]}{\left|#1\right|}
\providecommand{\norm[1]}{\abs{\abs{#1}}}%

\newcommand{\startmodif}{}
\newcommand{\stopmodif}{\color{black}}

\title{On the mean square error of randomized averaging algorithms\thanks{
The work of P.~Frasca was partly supported by the Italian Ministry of University and Research  under grant PRIN-20087W5P2K.
This paper presents research results of the Belgian Network DYSCO (Dynamical Systems, Control, and Optimization), funded by the Interuniversity Attraction Poles Programme, initiated by the Belgian State, Science Policy Office. The scientific responsibility rests with its authors.}} %

\author{Paolo Frasca\thanks{Dipartimento di Scienze Matematiche, Politecnico di Torino, corso Duca degli Abruzzi 24, 10129 Torino, Italy. paolo.frasca@polito.it} %
\and
Julien M. Hendrickx\thanks{ICTEAM Institute, Universit\'e catholique de Louvain, Avenue Georges Lemaitre 4, B-1348 Louvain-la-Neuve, Belgium. julien.hendrickx@uclouvain.be}}    %

\date{}

\begin{document}
      
      \maketitle
\begin{abstract}
This paper regards randomized discrete-time consensus systems that preserve the average \startmodif\quotes{on average}. \stopmodif
As a main result, we provide an upper bound on the mean square deviation of the consensus value from the initial average. Then, we \startmodif apply our result to systems where few or weakly correlated interactions  take place\stopmodif: these assumptions cover several algorithms proposed in the literature.
For such systems we show that, when the network size grows, the deviation tends to zero, \startmodif and the speed of this decay is not slower than the inverse of the size. \stopmodif
Our results are based on a new approach, \startmodif which is unrelated \stopmodif  to the convergence properties of the system.
%\rmv{: this independence questions the relevance in this context of the spectral properties of the matrices related to the graph of possible interactions, which have appeared in some previous results}.
\end{abstract}

%\tableofcontents

\section{Introduction}\label{sec:intro}
%\subsection*{Motivation}
In modern control and signal processing applications, effective and easy-to-implement distributed algorithms for computing averages are an important tool.
As a significant and \startmodif motivational \stopmodif example, we consider the problem of estimating the expectation of a random variable of interest. By the law of large numbers, the sample average is an unbiased estimator, and its mean square error decreases as the inverse of the number of samples \startmodif increases, provided the random variables have finite second moment. \stopmodif
In a distributed setting, the sample values are available at the nodes of a communication network, and the average needs to be approximated by running an iterative consensus system, \startmodif which has the sample data as the initial condition. Clearly, we have to ensure that along the iterations of the consensus system, no (or little) deviation from the correct average is introduced. 
However, a global property such as average preservation may be harder to satisfy when updates are performed asynchronously, unreliably or following a random scheme. \stopmodif In the case of stochastic updates, a weaker requirement is \startmodif the preservation of the expected average: \stopmodif such systems are known to converge to a consensus under mild conditions, but their consensus value is in general different from the average: \startmodif it is actually a random variable whose expected value is the initial average. \stopmodif
In this paper, we consider linear randomized asynchronous averaging algorithms, and we analyze the mean square deviation of the consensus value from the initial average. We want to ensure that this error is small, so that \startmodif the averages  are computed accurately. In particular, we aim to provide \stopmodif conditions under which the mean square error tends to zero when the number of samples, i.e.\ the number of nodes, grows. We will refer to this property \startmodif as the \stopmodif {\em accuracy} of the algorithm.
%conditions such that increasing the number of samples, that is of nodes in the network, leads to improve the mean square precision of the average as an estimator, just like in a centralized setting.

%\subsection*{Literature review}

The opportunity of using randomized algorithms to compute averages has already attracted a significant interest, as testified by 
recent surveys and special issues
~\cite{AGD-SK-JMF-MGR-AS:10,AS-MC-MG-JNT-MV:11}.
Convergence theories for randomized linear averaging algorithms have been developed \startmodif by several authors. A classic reference is~\cite{RC:86}, but more recently other conditions have been used in a few works including~\cite{FF-SZ:08a,ATS-AJ:08,ATS-AJ:10,IM-JSB:11,BT:12}. 
As we will formally define later, random linear averaging algorithms can be seen as the multiplication of \stopmodif the node-indexed state by a random update matrix.
In principle, the variance of the consensus value can be exactly computed by the formula in~\cite[Eq.~(7)]{ATS-AJ:10}, which involves the dominant eigenvectors of the first two moments of the update matrix. Unfortunately, little is known about these eigenvectors, and in particular explicit formulas are not available, so that these results are difficult to apply.
A few papers, on the other hand, have focused on specific examples of randomized algorithms, obtaining results which are interesting, although partial, from our perspective~\cite{FF-SZ:08b,TCA-MEY-ADS-AS:09,FF-PF:10a,FF-PF:10}. Typically, these results are obtained as a by-product of a convergence analysis and involve the eigenvalues of the update matrices, which are fairly well known for many families of communication graphs. 
\startmodif We will \stopmodif come back to these results in Section~\ref{sect:applications} when discussing some example algorithms.

%
%%
%Using these methods, asymptotical unbiasedness on general graphs has only been proved in~\cite{FF-SZ:08b} for the following simple algorithm: at each time step, a pair of neighbor nodes is selected: one node communicates her state to the other, which updates her own state to a convex combination, with weight $q$, of her previous state and the received state. In Section~\ref{sect:applications} we will come back to this example, and introduce a few others.

%\subsection*{Contribution}
In this paper, we consider discrete-time consensus systems with random updates that preserve the  \startmodif expected average\stopmodif, and we provide new bounds on the mean square deviation of the current average from the initial average.
We show that under certain conditions the expected increase of the deviation is bounded proportionally to the expected decrease of the disagreement. \startmodif We then obtain bounds \stopmodif on the total deviation which are proportional to the initial disagreement and, unlike previous results, are actually independent of the convergence properties: indeed they hold at \startmodif all times \stopmodif regardless of convergence.
Compared to those already available in the literature, our bounds typically result in less conservative \startmodif (and often more general) \stopmodif  estimates of the deviation error and, remarkably, they are independent of the global properties, \startmodif such as connectivity or graph spectrum and eigensystem of the communication network. Instead, only local network properties, like degree, play a role in the examples.  \stopmodif By contrast, we recall that %all known
results about convergence to consensus, and speed of convergence, depend on global network properties. 
Our estimates show that, \startmodif under weak assumptions on the update law, the deviation tends to zero when the number of nodes grows. This is true for \stopmodif
\begin{enumerate}[i)]
\item systems where few updates take place simultaneously; and 
\item systems where the updates have small statistical dependence across the network.
\end{enumerate}
Thanks to their generality and to their dependence on local network properties \startmodif only, \stopmodif our results offer effective and easy-to-implement guidelines to the designer who needs to choose a network and an algorithm to solve an estimation problem.

\subsection*{Notation and preliminaries}
The set of real numbers is denoted by $\real$, the set of nonnegative integers by $\integernonnegative$.
In this work, we use the notion of \emph{(weighted directed) graph}, which we define as a pair $G=(I,A)$, where $I$ is a finite set whose elements are called \emph{nodes} and $A\in \real^{I\times I}$ is a matrix with nonnegative entries. Resorting to more standard graph-theoretic jargon, we may equivalently think of an implicit edge set $E=\setdef{(i,j)\in I\times I}{A_{ij}>0}$. %\rmv{ and say that $i$ is connected to $j$ when $A_{ij}>0$.}
For simplicity, we will sometimes assume that a graph may have no loops, that is $A_{ii}=0$ for every $i\in I.$
%
%Any nonnegative matrix $A$ defines thus a weighted directed graph. 
%
%\margin{PF: (i)row-degree is not used. (ii) can't we possibly include the diagonal?}
%The \emph{row-degree} $\degrow_i$ of $i$ is the number of \jmh{off-diagonal} positive elements in the $i$-{th} row of $A$, i.e. the cardinality of $\{j\neq i :a_{ji}>0\}$. The \emph{column-degree} $\degcol_i$ of $i$ is the number of \jmh{off-diagonal} positive element in the $i$-{th} column of $A$, i.e. the cardinality of $\{j\neq i:a_{ij}>0\}$. 
%
%\rmv{The \emph{column-degree} $\degcol_i$ of $i$ is the number of off-diagonal positive elements in the $i$-{th} column of $A$, i.e. the cardinality of $\{j\neq i:a_{ij}>0\}$. }
%
%\rmv{The graph is said to be \emph{strongly connected} if for every node $i$ and $j$, there exists a sequence $i=i_0,i_1,\dots,i_p=j$ of nodes such that $A_{i_k,i_{k+1}}>0$ for $k=0,\dots,p-1$. }
%Intuitively, the graph is strongly connected if, when repeating the iteration $x(t+1) = Ax(t)$, every $x_i(0)$ eventually influences every other $x_j(t).$
% pf: isn't too soon for this dynamical intuition?
%
%\margin{I have included Laplacians here, in this general definition}
% It is known that if the graph is strongly connected, $\1$ is the only eigenvector associated to 0.
Given a graph, that is, a nonnegative matrix $A$, we can define an associated Laplacian matrix $L(A)\in \real^{I\times I}$ 
%as the matrix such that 
by
$[L(A)]_{ij} = -A_{ij}$ if $i\neq j$ and $[L(A)]_{ii} = \sum_{j:j\neq i}A_{ij}$. 
Observe that $L(A)$ is positive semidefinite and that $L(A)\1=0$, provided we denote by $\1$ the vector of suitable size whose components are all 1. Besides, to any matrix $L$ satisfying $L \1 = 0$ with nonpositive off-diagonal elements, one can associate a corresponding weighted graph. Finally, the conjugate transpose of the matrix $A$ is denoted by $A^*$, and inequalities $A\leq B$ between two matrices $A$ and $B$ denote the fact that $A-B$ is negative semi-definite.

%We consider weighted graphs $G=(I,E,A)$, where $E\subseteq I\times I$, and $A\in \real^{n\times n}$ satisfies $a_{ij}=[A]_{ij} >0$ if $(j,i) \in E$ and $a_{ij}=0$ is $(j,i)\not \in E$.  The reversal of indices is used so that the direction of the edges corresponds to the direction of the information in the iteration $x(t+1) = Ax$; the presence of $(i,j)$ indicates then that $x_i(t)$ influences the value of $x_j(t+1)$.  
%
%
%If the edge $(i,j)$ is connected, we say that $i$ is an \emph{in-neighbor} of $j$, and that $j$ is an \emph{out-neighbor} of $i$. The ro-degree of a node is its number of in-neighbors or number of edges arriving at the node, and its \emph{out-degree} is the number of out-neighbor, or number of edges leaving the node. 
%
%
%A \emph{directed path} from $i$ to $j$ is a sequence $i=i_0,i_1,\dots,i_p=j$ of node such that $(i_k,i_{k+1})\in E$ for $k=0,\dots,p-1$. A directed graph is \emph{strongly connected} if there is a directed path from $i$ to $j$ for every two nodes $i,j$.

%
%\subsection*{Outline}
%%We describe our problem formally and provide the result on which our analysis relies in Section \ref{sec:statement_and_result}. 
%
%%We finish by some concluding remarks and some discussion in Secition
%
%\jmh{Is an outline needed, we only have 4 sections. On the other hand we could split section 3 in subparts, but that would result in having 7 sections. What do you think?}
%\pf{I would say its not needed}

\section{Problem statement and main result}\label{sec:statement_and_result}
Given a set of nodes $I$ of finite cardinality $N$, we consider the discrete-time random process $x(\cdot)$ taking values in $\real^I$ and defined as
\begin{equation}\label{eq:system-cmptwise}
x_i(t+1) =  \sum_{j\in I} a_{ij}(t) x_j(t) \qquad \text{for all $i\in I$}, \quad t\in \integernonnegative,
\end{equation}
where for every $i,j\in I$, we assume $\{a_{ij}(t)\}_{t\in\integernonnegative}$ to be a sequence of independent and identically distributed random variables such that 
$a_{ij}(t)\ge0$ and $\sum_{\ell\in I}a_{i\ell}(t)=1$ for all $t\ge0$.
%
%where the $a_{ij}(t)$ are nonnegative random variables such that
%\begin{itemize}
%\item $\sum_{j\in I}a_{ij}(t)=1$ for all $i\in I$ and $t$;
%\item the distribution of $a_{ij}(t)$ is the same for all $t$;
%\item $a_{ij}(t)$ and $a_{hk}(s)$ are {\blue independent} whenever $t\neq s$.\margin{need INDIPENDENCE?}
%\end{itemize}
%
System~\eqref{eq:system-cmptwise} is run with the goal for the state of each node to provide a good estimate of the initial average $\frac1N\sum_{i\in I}x_i(0)$. 
Note that $x(0)$ is unknown but given, and that all our results will be valid for any $x(0)\in \real^I$.
System~\eqref{eq:system-cmptwise} can also be conveniently rewritten as 
\begin{equation*}x_i(t+1) =  x_i(t) + \sum_{j\in I} a_{ij}(t) (x_j(t)-x_i(t)) \qquad \text{for all $i\in I$}, \quad t\in \integernonnegative,
\end{equation*}
or in matrix form as
\begin{equation}\label{eq:system}
x(t+1) = x(t) - L(t)x(t) \qquad t\in \integernonnegative,
\end{equation}
where the matrix $L(t)$ is defined so that
$L_{ij}(t) = -a_{ij}(t)$ if $i\neq j$ and $L_{ii}(t) = \sum_{j:j\neq i}a_{ij}(t)$. 
Namely, $L(t)$ is the Laplacian matrix of a weighted graph $(I,A(t))$ where the entries of $A(t)$ are defined as
$[A(t)]_{ij}=a_{ij}(t)$. %if $i\neq j$ and zero otherwise. 
The convergence of~\eqref{eq:system} has been addressed in the literature: 
rather than in convergence, in this paper we are interested in the quality of the convergence value, in terms of its distance from the initial average.
For our convenience, we denote the average of the $x_i(t)$'s by 
$$
\bar x(t) %= \frac{1}{N}\1^* x(t) 
= \frac{1}{N}\sum_{i\in I} x_i(t)
 $$ and
we note that the average evolves according to
$\barx(t+1)=\barx(t)-\1^* L(t) x(t).$
The expected evolution of $\barx(t)$, conditional on the previous state, is written as \startmodif $\Exp[\barx(t+1)|x(t)]$. \stopmodif
Since under our assumptions $L(t)$ is independent from $x(t)$, we immediately deduce that 
$\Exp[\barx(t+1)|x(t)]=\barx(t)$ if and only if $\1^*\Exp[L(t)]=0.$
In view of this fact,
we restrict our attention to systems that preserve the  \startmodif expected average\stopmodif, that is, we will assume $\1^* \Exp[L(t)]=0$, implying that $$\Exp[\barx(t)]=\barx(0) \quad \text{for all $t\ge 0$}.$$
Consequently, we are left with the problem of studying the variance of $\barx(t)$, that is $\Exp[\left(\barx(t)-\barx(0)\right)^2]$.
We will derive all our bounds from the following general result.
%which establishes that, under some conditions, the increase of the deviation is bounded proportionally to the decrease of the disagreement.
%
%\comjh{Giacomo suggested this idea of phrasing it as a conservation result. Should we somehow acknowledge that? } 
% Let us keep that in mind and decide at the end.
% \comjh{Please check the probabilistic part even more carefully}
% I will also re-read it at the end
%
For $y\in \real^I$, we denote $\bar y=\frac1N\sum_i y_i$ and $V(y) = \frac1N \sum_{i}\big(y_i-\bar y\big)^2$.
\begin{theorem}[Accuracy condition]\label{th:main-condition}
Let $x$ be an evolution of system~\eqref{eq:system}. %, and denote 
%$V(t)=\frac1N \sum_{i}\big(x_i(t)-\barx(t)\big)^2$
%$V(y) = \frac1N \sum_{i}\big(y_i-\bar y\big)^2$ for $y\in \real^I$.
%\barxsq(t) = \frac{1}{N}\sum_i y_i^2$ for all $y\in \real^N$. 
If $\1^*\Exp[L(t)]=0$ and there exists $\gamma>0$ such that
\begin{equation}\label{eq:condition-gamma}
\Exp[L(s)^*\1\1^*L(s)]\le \gamma\, \Exp[L(s)+L(s)^*-L(s)^*L(s)],\end{equation}
 then for every $t\ge0$, there holds
$$\Exp[(\bar{x}(t)-\bar{x}(0))^2]\le \frac\gamma{N+\gamma} V(x(0)).$$ 
If moreover the system converges to consensus ($x(t) \to x_\infty \1$, for $x_\infty \in \real$), then $\Exp\l[(x_\infty-\barx(0))^2\r]\leq \frac\gamma{N+\gamma}V(x(0)).$  
\end{theorem}
Note that $\frac{\gamma}{\gamma + N}$ is increasing with $\gamma$: it is close to $\frac\gamma{N}$ for small values of $\gamma$, and close to $1$ for large ones. The expected square error is thus always bounded by the initial disagreement when a valid $\gamma$ can be found. And when a $\gamma$ can be found which is independent of $N$, then the algorithm is {\em accurate}, according to the definition stated in the Introduction.
%
%The bound provided in Theorem~\ref{th:main-condition} is of course more conservative than the exact characterizations of the expected square error derived in~\cite[Eq.~(7)]{ATS-AJ:10}. However, it presents the main advantage of being easy to use. Indeed, we will see in the next section that generic expressions of $\gamma$ can be obtained for large classes of systems.
%
%In view of the motivating problem considered in the Introduction, we are interested in sequences of averaging algorithms, acting on networks of increasing size $N$.
%
The bound of Theorem \ref{th:main-condition} may of course be conservative compared with the exact characterizations of the expected square error derived in~\cite{ATS-AJ:10}, but it presents the main advantage of being easy to use. Indeed, we will see in the next section that general expressions of $\gamma$ can be obtained for large classes of systems.
%
%When $\gamma$ can be found which is independent of $N$, then the algorithm is {\em accurate}, according to the definition stated in the Introduction.
%
\begin{proof}
We define $
C(y) := N(\gamma+N) {\bar y}^2  + N\gamma V(y),
$
for all $y\in \real^I$, a linear combination of the square average value and the disagreement\footnote{\startmodif The authors wish to thank Giacomo Como for suggesting to formulate the proof in terms of this quantity $C(y)$.\stopmodif}, with a ratio $\frac{\gamma}{\gamma+N}$ between the weights, with the intent \startmodif to show that \startmodif the expectation of $C(x(t))$ is nonincreasing. \stopmodif We begin by developing a simpler expression for $C(y)$. Observe that $\bar y = \frac{1}{N}\1^*y$, and that $V(y) = \frac{1}{N}\sum_{i=1}^N(y_i-\bar y)^2 = \frac{1}{N}(y-\frac{1}{N}\1\1^*y)^*(y-\frac{1}{N}\1\1^*y)$. Therefore,
$$
C(y) = y^*\prt{\frac{N(N+\gamma)}{N^2}\1\1^* + \frac{N\gamma}{N}\prt{I-2\frac{1}{N}\1\1^*+\frac{1}{N^2}\1\1^*\1\1^*}}y = y^*\prt{\1\1^* + \gamma I}y.
$$
We now show that $\Exp(C(y-Ly)) \leq C(y)$ for any $y\in \real^I$, where the Laplacian $L$ is a random variable having the same distribution as $L(t)$. We can express the difference as
 $$
C(y-Ly) -C(y) = - y^*L^*(\1\1^* +\gamma I) y - y^*(\1\1^* +\gamma I) Ly + y^*L^*(\1\1^* +\gamma I)Ly.
$$
Since it is assumed that $\1^*\Exp L =0$, we have then 
\begin{equation}\label{eq:C(x-LX)-C(x)}
\Exp[C(y-Ly) - C(y)] = - y^*\Exp\prt{-\gamma L^* -\gamma L + \gamma L^*L + L^*\1\1^*L }y  \leq 0,
\end{equation}
where the last inequality follows from the assumption in~\eqref{eq:condition-gamma}. Equation \eqref{eq:C(x-LX)-C(x)} implies that if $x(t)$ follows the process~\eqref{eq:system}, then $\Exp[C(x(t+1))| x(t)] \leq C(x(t))$. As a result, if $\bar x(0) = 0$ there holds
$$
N(\gamma+N)\Exp[(\bar x(t))^2]  + N\gamma\, \Exp [V(x(t))] \leq N\gamma V(x(0)),
$$
and thus $\Exp(\bar x(t))^2 \leq \frac{\gamma}{N+\gamma}V(x(0))$ since $V(x(t))\geq 0$, which proves the result in that case. Otherwise, the result is obtained by applying the previous inequality to the translated system $x(t) - \bar x(0)\1$. 
\end{proof}

%%%%%%%%%%%%%%%%%%%%%%

\section{Applications and examples}\label{sect:applications}
In this section we see classes of systems of type~\eqref{eq:system} for which we can apply Theorem~\ref{th:main-condition}, that is, we can find $\gamma$ satisfying~\eqref{eq:condition-gamma}. 
Before presenting these example systems, we prove a general lemma which simplifies the search for~$\gamma$: indeed, the proofs of our results will involve estimating
$\Exp(L^*\1\1^*L)$ and $\Exp(L^*L)$ in terms of $\Exp(L+L^*)$,
where we remind that an inequality between two matrices $A\le B$ is intended as $A-B$ being negative semidefinite.
Before the general lemma, we need the following preliminary result.
\begin{lemma}\label{lem:sum_square}
Suppose that the coefficients $c_1,\dots,c_m$ are nonnegative. Then, there holds
$$\prt{\sum_{i=1}^mc_iz_i}^2\leq  \prt{\sum_{i=1}^mc_i}\sum_{i=1}^m c_i z_i^2$$
\end{lemma}
\begin{proof}
Let $u,v\in \real^m$ be defined by $u_i = \sqrt{c_i}$ and $v_i = \sqrt{c_i}z_i$. It follows from Cauchy-Schwartz inequality that
$$
\prt{\sum_{i=1}^mc_iz_i}^2 = \prt{u^*v}^2\leq \prt{\norm{u}_2\norm{v}_2}^2= \prt{\sum_{i=1}^mu_i^2}\prt{\sum_{i=1}^m v_i^2} =  \prt{\sum_{i=1}^mc_i}\sum_{i=1}^m c_i z_i^2.
$$ 
\end{proof}

\begin{lemma}[Laplacian bounds]\label{lemma:Laplacian-inequality}
Let $L$ be the Laplacian of a weighted directed graph with weight matrix $A$, define $a_{ii}:=1-\sum_{j\neq i}a_{ij}$, and let $\admin>0$ be such that $a_{ii}\ge\admin$ for all $i\in I$.
\newcounter{enumii_saved} 
\begin{enumerate}[(i)]
\item\label{item:was-8b} If $\1^*L=0$, then
\be\label{eq:was-8b}
%L^*L\leq \armax (L+L^*).
L^*L\leq (1-\admin) (L+L^*).
\ee
\setcounter{enumii_saved}{\value{enumi}}
\end{enumerate}

Let now $L$ be a random \startmodif matrix \stopmodif such that the lower bound $\admin$ is valid almost surely.
\begin{enumerate}[(i)]
\setcounter{enumi}{\value{enumii_saved}}
\item\label{item:was-8c} If $\1^*\Exp(L)=0$, then 
\be\label{eq:was-8c}\Exp(L^*L)\leq (1-\admin) \Exp(L+L^*).
\ee %where the upper bound $\armax$ has to be valid for all realisations.
\item\label{item:was-9} If $\1^*\Exp(L)=0$ and there exists $\beta >0$ such that $$\Exp(L^*\1\1^*L) \leq \beta\,\Exp(L+L^*),$$ then 
$ \Exp[L^*\1\1^*L]\le \gamma\, \Exp[L+L^*-L^*L]$ holds
for $\dst\gamma = \frac{\beta}{\admin}.$
\end{enumerate}
\end{lemma}
%Now we come back to the proof of Lemma~\ref{lemma:Laplacian-inequality}.
\begin{proof}%[Proof of Lemma~\ref{lemma:Laplacian-inequality}]
%\margin{done very small changes in the proof}
\startmodif Let $y\in \real^I$ be arbitrary but fixed. \stopmodif
To prove claim~(\ref{item:was-8b}), we note that $(Ly)_i = \sum_{j} a_{ij}(y_i-y_j)$ and therefore $y^*L^*Ly = \sum_{i}\big(\sum_{j:j\neq i}a_{ij}(y_i-y_j)\big)^2.$ For every $i$, \startmodif since $1-\admin\geq \sum_{j:j\neq i}a_{ij}$, Lemma~\ref{lem:sum_square} implies \stopmodif that  $$\startmodif (Ly)_i^2 =\stopmodif \Big(\sum_{j:j\neq i}a_{ij}(y_j-y_i)\Big)^2 \leq (1-\admin)\sum_{j:j\neq i} a_{ij}\prt{y_j-y_i}^2,$$
and by summing on $i$ that 
\be\label{eq:was-8a}
y^*L^*Ly\leq (1-\admin) \sum_i \sum_{j:j\neq i}a_{ij}(y_i-y_j)^2.\ee
%%
%%holds, independently of the assumption that $\1^*L=0$. 
%
%
%Also note that 
%%\begin{align}
%$$\sum_i \sum _{j:j\neq i} a_{ij}(y_j-y_i)^2 = \sum_i y_i^2 {\sum _{j:j\neq i} a_{ij}}  - 2 \sum_i \sum _{j:j\neq i}a_{ij}y_iy_j + \sum_j y_{j}^2 \sum _{i:i\neq j} a_{ij}.
%$$
%Since $1^*L =0$, we have $\sum_{i:i\neq j} a_{ij} = \sum_{j:j\neq i}a_{ij}$. Therefore, a relabeling of the third term leads to
%$$\sum_i \sum _{j:j\neq i} a_{ij}(y_j-y_i)^2 = 2 \sum_i \sum _{j:j\neq i} a_{ij}\,y_i^2 - 2 \sum_i \sum _{j:j\neq i}a_{ij}y_iy_j = 2y^*Ly = y^*(L+L^*)y,
%$$
%where we have used  $y^*L^*y=(y^*Ly)^*= y^*Ly$. Using~\eqref{eq:was-8a}, 
%statement~(\ref{item:was-8b}) follows.
%
\startmodif
Statement~(\ref{item:was-8b}) then follows by noting that 
$\sum_i \sum _{j:j\neq i} a_{ij}(y_j-y_i)^2= y^*(L+L^*)y$ because $\1^*L=0$.
We now prove \stopmodif statement~(\ref{item:was-8c}). It follows from~\eqref{eq:was-8a} that 
$$y^*\Exp(L^*L)y=\Exp(y^*L^*Ly) \leq \Exp\left[(1-\admin) \sum_i \sum_{j:j\neq i}a_{ij}(y_i-y_j)^2\right] = (1-\admin)\sum_i \sum_{j:j\neq i}\Exp(a_{ij})(y_i-y_j)^2.$$
Since $ \Exp(L)$ is a (deterministic) Laplacian and $1^*\Exp(L)=0$, we can apply the same argument leading to~\eqref{eq:was-8b} in order to argue that 
$$ y^*\Exp(L^*L)y \leq (1-\admin)y^* \Exp(L+L^*)y,$$
which implies~\eqref{eq:was-8c}.
Finally, we prove the last claim~(\ref{item:was-9}).
It follows from~\eqref{eq:was-8c} that $-(1-\admin)\Exp(L+L^*)\leq -\Exp(L^*L)$. Therefore, the existence of $\beta$ implies that for $\gamma = \frac{\beta}{\admin}$, there holds
$$
\Exp(L^*\1\1^*L) \leq \beta\,\Exp(L+L^*)\leq \gamma\,\Exp(L+L^*)- \gamma\,(1-\admin)\Exp(L+L^*)\leq \gamma\,\Exp\prt{L+L^* - L^*L}. 
$$
\end{proof}

When we apply Lemma~\ref{lemma:Laplacian-inequality} to a system of type~\eqref{eq:system}, the quantity $\admin$ is \startmodif in fact \stopmodif a lower bound on the ``self-confidence'' $a_{ii}(t)$ of the nodes. For a constant $\beta$, the bound on the mean square error is thus inversely proportional to the minimal self-confidence. This remark is consistent with the intuition that, when $a_{ii}(t)$ is very small, the information held by some nodes may be almost entirely \quotes{forgotten} in one iteration, possibly resulting in large variations of the average.

\subsection{Limited simultaneous updates}

In this section, we show that \startmodif a scalar $\gamma$ which satisfies \stopmodif the condition in Theorem~\ref{th:main-condition} can be found when the number, or at least the contribution, of the simultaneous updates is small.
The next result has the following interpretation: the mean square deviation can be bounded proportionally to the ratio between ``strength'' of the interactions in the system and the ``self-confidence'' of each node.
%To ensure a small deviation, $\gamma$ has to be larger than the ratio between the total ``strength'' of the interactions and the ``self-confidence'' of each node.
Note that from now on, when studying the evolution of system~\eqref{eq:system}, we will for brevity \startmodif avoid \stopmodif to write the dependence on time of the random variables $a_{ij}$ and $L$, if this causes no confusion.

\begin{theorem}[Limited updates]\label{th:counting}
Consider system~\eqref{eq:system} and let $\Amax$ and $\admin$ be two positive constants such that almost surely $\sum_i \sum_{j:j\neq i} a_{ij}\le \Amax$ and $a_{ii}\ge \admin$ for all $i\in I$.
If $\1^*\Exp(L)=0$, then the condition of Theorem~\ref{th:main-condition} holds for  
$$
\gamma=\frac{\Amax}{\admin}.
$$
\end{theorem}
\begin{proof}
It follows from Lemma \ref{lem:sum_square} that 
$$
%y^*L^*\1\1^*Ly  = \prt{\sum_i\sum_{j\neq i}a_{ij}(y_j-y_i)}^2 \leq \Emax\sum_i\sum_{j\neq i}\prt{a_{ij}(y_j-y_i)}^2 \leq\Emax\amax\sum_i\sum_{j\neq i}a_{ij}\prt{y_j-y_i}^2.
y^*L^*\1\1^*Ly  = \Big(\sum_i\sum_{j:j\neq i}a_{ij}(y_j-y_i)\Big)^2 \leq \Amax\sum_i\sum_{j:j\neq i}a_{ij}\prt{y_j-y_i}^2.
$$
Therefore,
$$
\Exp\prt{y^*L^*\1\1^*Ly} \leq \Amax\sum_i\sum_{j:j\neq i} \Exp(a_{ij})\prt{y_j-y_i}^2=\Amax  y^T\Exp(L+L^*)y,
%\Exp\prt{y^*L^*\1\1^*Ly} \leq \Emax\amax\sum_i\sum_{j\neq i} \Exp(a_{ij})\prt{y_j-y_i}^2=\Emax\amax  y^T\Exp(L+L^*)y,
$$
where we have used Lemma~\ref{lemma:Laplacian-inequality}(\ref{item:was-8b}), so that $\Exp\prt{L^*\1\1^*L} \leq\Amax\Exp[L+L^*]$. The result follows from Lemma~\ref{lemma:Laplacian-inequality}(\ref{item:was-9}).
\end{proof}

Theorem~\ref{th:counting} can be applied to several particular cases involving small number of edges or small interactions: we discuss here two of them, drawn from the literature.
%\rmv{The following is the simplest example of a randomized averaging algorithm.}

\begin{example}[Asynchronous Asymmetric Gossip Algorithm (AAGA)]\label{ex:AGA}
Let a graph $G=(I,W)$ and $q\in (0,1)$ be given, such that $\1^*W\1=1$.
For every $t\ge 0$, one edge $(i,j)$ is sampled from a distribution such that the probability of selecting $(i,j)$ is $W_{ij}$. Then, 
$$x_i(t+1)=(1-q)\,x_i(t)+q\,x_j(t),$$ and $x_k(t+1)=x_k(t)$ for $k\neq i$. 
\end{example}
Observe that if $W\1=W^*\1$, then $\1^*\Exp[L(t)]$ holds for the AAGA, and we can apply Theorem~\ref{th:counting} with $\Amax=1-\admin=q$ since only one node is sending her state to another. This leads to $\gamma = \frac{q}{1-q}$, meaning that the expected deviation of the asymptotic value is not larger than $\frac1N\frac{q}{1-q+\frac{q}{N}}V(x(0)).$
The AAGA system is also studied in~\cite[Section~4]{FF-SZ:08b}: the authors prove, assuming that the components of $x(0)$ are i.i.d. random variables with variance $\sigma^2$, that the square deviation is not larger than 
% where it is proved that the mean square deviation of the limit value is not larger than 
$\frac{q-\frac qN}{1-q+\frac qN} \frac1N \sigma^2.$% if $x(0)$ was a random variable whose components are i.i.d. with variance $\sigma^2$. 
Taking into account that the expected value of $V(x(0))$ is $\prt{1-\frac{1}{N}}\sigma^2$ in that case,  we see that our bound allows retrieving their result.

The next example, which applies very naturally to wireless networks, has attracted a significant attention~\cite{FF-SZ:08a,TCA-MEY-ADS-AS:09,AGD-SK-JMF-MGR-AS:10}.
\begin{example}[Broadcast Gossip Algorithm (BGA)]\label{ex:BGA}
Let a graph $G=(I,W)$ and $q\in (0,1)$ be given, such that $W\in \{0,1\}^{I\times I}$.
For every $t\ge 0$, one node $j$ is sampled from a uniform distribution over $I$. Then, 
$x_i(t+1)=(1-q)\,x_i(t)+q\,x_j(t)$ if $W_{ij}>0$ and $x_i(t+1)=x_i(t)$ otherwise. In other words, one randomly selected node broadcasts her value to all her neighbors, which update their values accordingly.
\end{example}
Previous results about the deviation of BGA are dependent on the topology of the network.
In~\cite[Proposition~3]{TCA-MEY-ADS-AS:09} it is proved that the expected square deviation is upper bounded by $$\dst V(x(0)) \l(1-\frac{\lambda_1}{\lambda_{N-1}} \frac1{1-\frac12\frac qN \lambda_{N-1}}\r),$$ where $\lambda_i$ is the $i$-th smallest non-zero eigenvalue of the Laplacian of the graph $G$. In~\cite[Proposition~3.3]{FF-PF:10a} the authors obtain the upper bound $\dst 2 V(x(0))\frac{q}{1-q}\frac{\degmax^2}{N\lambda_1},$ where $\degmax$ is the maximum degree of the graph. None of these bounds suffices to show that the deviation goes to zero when $N$ grows: for instance, $\frac{\degmax^2}{N\lambda_1}\ge\frac{N}{\pi^2}$ holds on a cycle graph. Accuracy is shown for cycles and some other sequences of graphs  in~\cite{FF-PF:10}, using Markov chain theory results from~\cite{FF-JCD:10},  but a general proof of accuracy is not available in the literature.
Based on simulations, it was however conjectured in~\cite{FF-PF:10a} that the mean square error of the BGA is proportional to the ratio between the degree and the number of nodes. This fact can actually be proved by applying Theorem~\ref{th:counting}, assuming that $W\1=W^*\1$.
Indeed, when $W\1=W^*\1$, there holds $\Exp[L(t)]=\frac{q}{N} L(W)$ (where we remind the reader that $L(W)$ is the Laplacian matrix corresponding to the weighted adjacency matrix $W$), and thus $\1^*\Exp[L(t)]=0$. Observe moreover that $\admin=1-q$ and $\Amax=q \degmaxcol$, since one node may send her value to at most $\degmaxcol$ neighbors. Theorem~\ref{th:counting} implies then that $\frac{q}{1-q}\degmaxcol$ is a valid value of $\gamma$, and a bound proportional to $\frac{\degmaxcol}{N}$ follows then from Theorem \ref{th:main-condition}.
Finally, since every system admits a trivial $\Amax= N$, Theorem \ref{th:counting} also implies that a valid $\gamma$ exists as soon as there is a $\admin>0$ for which $a_{ii}\ge\admin$ holds for all $i$. It follows then from Theorem \ref{th:main-condition} that the expected square error is bounded by the initial disagreement in all these cases. On the other hand, the AAGA system with two nodes and $q=1$, for which there is no such $\admin$, is an example of system for which no valid $\gamma$ exists.

\subsection{Uncorrelated updates}
In this section we show that a small $\gamma$ can still be found even if there are many simultaneous updates, provided that the correlation between the updates is sufficiently small. 
%\margin{Do you think this introduction is clear enough?}
The next result considers three cases: (a) all update coefficients are uncorrelated, (b) nodes update their value according to any stochastic scheme, but their decisions of update are uncorrelated to that of the other nodes, (c) nodes transmit their values 
%\rmv{(with some weights)} 
according to any stochastic scheme, but their decisions of transmission are uncorrelated to that of the other nodes.

\begin{theorem}\label{thm:indep}
Consider system~\eqref{eq:system} and let $\aindmax,\armax,\acmax,\admin$ be positive constants  such that $a_{ij} \le \aindmax $ (with $i\neq j$), $\sum_{j:j\neq i} a_{ij}\le \armax$, $\sum_{i:i\neq j}a_{ij} \le \acmax$, and $a_{ii}\ge\admin$ respectively hold almost surely. Suppose that $\1^*\Exp(L)=0$. The following implications about the value of $\gamma$ in Theorem \ref{th:main-condition} hold true.
\begin{enumerate}[(a)]
\item Uncorrelated coefficients: If all $a_{ij}$'s are uncorrelated, then $\gamma = \frac{\supscr{a}{ind}_{\max}}{\admin}$.
\item Uncorrelated updates: If $a_{ij}$ and $a_{kl}$ are uncorrelated when $i\neq k$, then $\gamma= \frac{\armax}{\admin}.$
\item Uncorrelated transmissions: If $a_{ij}$ and $a_{kl}$ are uncorrelated when $l\neq j$, then $\gamma= \frac{\acmax}{\admin}.$
\end{enumerate}
\end{theorem}
Note that (b) implies that any scheme (preserving \startmodif the expected average\stopmodif), where nodes update their values independently and have a minimal self-confidence, is accurate.
%\margin{General comment, should we work with y or x? + is the L notation OK?}
%\vspace{-.5cm}
\begin{proof}
We begin by proving (b), bouding $\Exp(L^*\1\1^*L)$ proportionally to $\Exp(L+L^*)$ in order to apply Lemma \ref{lemma:Laplacian-inequality}. Since $\1^*\Exp L = 0$, observe that $\Exp(L^*\1\1^*L) =\Exp(L^*\1\1^*L) - \Exp(L^*)\1\1^*\Exp(L)$. Besides, $\1^*Ly = \sum_{i,j}a_{ij}(y_j-y_i)$ holds for $y\in \real^I$. Therefore, we have
\begin{equation}\label{eq:dvlp_L11L}
y^*\Exp(L^*\1\1^*L)y =  \sum_{i,j,k,l}\prt{\Exp[a_{ij}a_{kl}]- \Exp a_{ij} \Exp a_{kl}}(y_j-y_i)(y_l-y_k).
\end{equation}
According to assumption (b), \startmodif if $i\neq k$, then \stopmodif $a_{ij}$ and $a_{kl}$ are uncorrelated so that $\Exp(a_{ij}a_{kl})= \Exp a_{ij}\Exp a_{kl}.$ We have then 
\begin{align*}
y^*\Exp(L^*\1\1^*L)y  &= \sum_{i,j,l} \Exp[a_{ij}a_{il}] (y_j-y_i)(y_l-y_i) - \sum_{i,j,l} \Exp a_{ij} \Exp a_{il} (y_j-y_i)(y_l-y_i)\\&= \Exp\Big[\sum_i \Big(\sum_j a_{ij}(y_j-y_i)\Big)^2\Big] - \sum_i \Big(\sum_j \Exp a_{ij}(y_j-y_i)\Big)^2.
\end{align*}
The second term in the last expression is clearly non-positive. Applying Lemma \ref{lem:sum_square} for each $i$ in the first term leads then to
\begin{equation}\label{eq:bound_indep_row}
y^*\Exp(L^*\1\1^*L)y  \leq \sum_i \Exp\Big[ \Big(\sum_{j\neq i} a_{ij}\Big)\Big(\sum_{j} a_{ij}\prt{y_j-y_i}^2\Big)\Big] \leq \armax \Exp\Big[\sum_{i,j} a_{ij}\prt{y_j-y_i}^2\Big] = \armax y^*\Exp\prt{L+L^*}y,
\end{equation}
where we have used the definition of $\armax$ and Lemma~\ref{lemma:Laplacian-inequality}(i). The result~(b) follows then from Lemma~\ref{lemma:Laplacian-inequality}~(iii).
Part (c) of the result is obtained in a parallel way, using $\Exp(a_{ij}a_{kl})= \Exp a_{ij} \Exp a_{kl}$ when $j\neq l$ instead of $i\neq k$ after equation \eqref{eq:dvlp_L11L}, and $\acmax$ instead of $\armax$ in equation \eqref{eq:bound_indep_row}. For part (a), one has $\Exp(a_{ij}a_{kl})= \Exp a_{ij} \Exp a_{kl}$ unless $i=k$ and $j=l$. Therefore equation \eqref{eq:dvlp_L11L} becomes $$y^*\Exp(L^*\1\1^*L)y = \sum_{i,j} \Exp[a_{ij}^2](y_j-y_i)^2 - \sum_{i,j}  (\Exp a_{ij})^2(y_j-y_i)^2\leq \aindmax \Exp \Big[\sum_{i,j} a_{ij}^2(y_j-y_i)^2\Big],$$ which \startmodif allows us to conclude \stopmodif using again Lemma \ref{lemma:Laplacian-inequality} (i) and (iii). 
\end{proof}

The following is natural example of uncorrelated updates.
\begin{example}[Synchronous Asymmetric Gossip Algorithm (SAGA)]\label{ex:SAGA}
Let $q\in (0,1)$ and a graph $G=(I,W)$ be given, such that $W\1=\1$.
For every $t\ge 0$, and every $i\in I$ one edge $(i,j_i)$ is sampled from a distribution such that the probability of selecting $(i,j_i)$ is $W_{i,j_i}$. Then, for every $i\in I$,
$x_i(t+1)=(1-q)\,x_i(t)+q\,x_{j_i}(t).$ In other words, every node chooses one neighbor, reads her value, and updates her own value accordingly.
\end{example}

Previous results on SAGA are only able to guarantee accuracy on certain sequences of graphs:
%\begin{remark}[Earlier results on SAGA]\label{rem-early-SAGA}
\newcommand{\esr}{\textup{esr}}
in~\cite[Section~5]{FF-SZ:08b}, the authors derive an upper bound on the deviation of the limit value, which for symmetric $W$ and large $N$ is asymptotically equivalent to $ \frac{q}{1-q}  \frac1{2N} \frac{1}{1-\esr(W)} V(x(0)),$ where $\esr(W)$ is the second-largest absolute value of the eigenvalues of $W$. This result fails to prove accuracy for some sequences of graphs: for instance, on a cycle graph with positive $W_{ij}$s equal to $1/2$, we have
$\frac1{2N} \frac{1}{1-\esr(W)}=
 \frac1{2N} \frac{1}{\cos\left(\frac{2 \pi}{N}\right)} V(x(0))\ge \frac{q}{1-q} \frac{N}{4\pi^2} V(x(0))$.
Our approach allows proving asymptotic accuracy independently of the topology of the networks, provided that $W$ is such that $\1^*W=\1^*$. Observe indeed that $\Exp[L(t)]=q L(W)$, and thus that $\1^*\Exp[L(t)]=0$. Moreover, since every node receives information from exactly one neighbor, there holds $\armax =q$ and $\admin=1-q$. Since the choices of neighbors are independent, we can apply Theorem \ref{thm:indep}(b) to show that $\frac{q}{1-q}$ is a valid value of $\gamma$, so that the expected square deviation is bounded by $\frac{q}{1-q+\frac{q}{N}} V(x(0))$, as in the case of the AAGA.

\subsection{Simultaneous correlated updates}

We have seen 
%in the previous subsections 
that accurate systems are obtained when there are few simultaneous updates or when the updates are uncorrelated. When these two conditions are not met, one can have systems whose expected square deviation is large with respected to $V(x(0))$, or does not decrease when $N$ grows. 
However, one should not conclude that every system with unbounded and not strictly uncorrelated updates must not be accurate. In particular, small mean square errors can still occur for systems where the updates follow a probability law involving some partial correlations. An example is the following algorithm, which generalizes the BGA and has been proposed in~\cite{TCA-ADS-AGD:09}.

\begin{example}[Probabilistic Broadcast Gossip Algorithm (PBGA)]\label{ex:PBGA}
Let $q\in(0,1)$ and $G=(I,W)$. At each time step, one node $j$, sampled from a uniform distribution over $I$, broadcasts her current value. Every node $i$ receives the value with a probability $W_{ij}\in[0,1]$.
%In particular, $W_{ij}=0$ if there is no communication link between $i$ and $j$. 
When node $i$ does receive the value from $j$, she updates her value to $x_i(t+1) = x_i(t) + q(x_j(t)-x_i(t))$. Otherwise, $x_k(t+1)=x_k(t)$.
\end{example}

%In the case of the PBGA, we have
%$$L(t)=- q \sum_{j\in I} \ind{\{j\:\text{broadcasts}\}} \sum_{i\in I} \ind{\{i\: \text{receives from}\:j\}}(e_i e_j^*-e_i e_i^*),$$
%where $0<q<1$ and we assume that $\Pr[\{i\: \text{receives from}\:j\}]=W_{ij}$ and $\Pr[{\{j\:\text{broadcasts}\}}]=\frac1N.$

%For simplicity and because in the applications $W_{ij}$ is supposed to depend on the physical distance between nodes $i$ and $j$, we assume $W_{ij}=W_{ji}.$

\begin{proposition}[PBGA is accurate]\label{prop:PBGA}
%If $W=W^*$Theorem~\ref{th:main-condition} holds with $\gamma= (W_{\max}+1)  \frac{q}{1-q},$
Assume that $W=W^*$. Then, Theorem~\ref{th:main-condition} holds with
$\gamma=(W_{\max}+1)  \frac{q}{1-q},$
where $W_{\max}=\max_{i\in I}\sum_{j\in I} W_{ij}.$
\end{proposition}
%\vspace{-.5cm}
\begin{proof}
%\margin{proof merged with lemma, removed useless claims}
From~\cite[Lemma~2]{TCA-ADS-AGD:09} we can quickly derive that for every $t\ge0$,
\begin{align*}
&\Exp[L(t)]=\frac{q}N L(W)\\
&\Exp[L(t)^*L(t)]=2 \frac{q^2}N L(W)\\
&\Exp[L(t)^*\1\1^*L(t)]=\frac{q^2}N L(W)^2 + 2 \frac{q^2}N L(W) - 2\frac{q^2}N L(W\cdot W),
\end{align*}
where $W\cdot W$ denotes entrywise product. 
%\add{and we remind you that $L(.)$ denotes the Laplacian associated to a matrix}
The assumption on $W$ implies that $\1^*\Exp[L(t)]=0$, and in order to apply Theorem~\ref{th:main-condition}  we have to find $\gamma$ which satisfies the inequality
\begin{equation*}\frac{q^2}N L(W)^2+ 2 \frac{q^2}N L(W) -2\frac{q^2}N L(W\cdot W)\le \gamma \left(2\frac{q}{N} L(W) - 2\frac{q^2}N L(W)\right),\end{equation*}
that is $$ L(W)^2-2 L(W\cdot W)\le 2\prt{\gamma \frac{1-q}{q}-1} L(W).$$ Since any Laplacian --and in particular $L(W\cdot W)$-- is positive semidefinite, a sufficient condition for the previous inequality to hold is 
$$ L(W)^2\le 2\prt{\gamma \frac{1-q}{q}-1}  L(W).$$
\startmodif Gershgorin's disk Lemma implies that the spectral radius of $L(W)$ is not larger than $2 W_{\max}$, and the result follows.
\end{proof}
%
%\margin{skip this remark?}
%Observe that the BGA system of Example~\ref{ex:BGA} is a particular case of the PBGA: applying Proposition~\ref{prop:PBGA} to BGA leads to a value $\gamma = (\degmaxcol+1)\frac{q}{1-q}$, which is more conservative than the value $\gamma =\degmaxcol\frac{q}{1-q}$ obtained in Corollary~\ref{cor:BGA}.

%However, one should not conclude that every class of systems with unbounded and not strictly uncorrelated updates is necessarily asymptotically biased. In particular, small mean square errors can still occur for systems where the updates follow some more complex probability law, presenting some partial correlations. In the context of this research, we haave encountered several such classes of systems, which we chose not to present here for the sake of concision. 

\section{Conclusion and perspectives}
We have developed a new way of evaluating the mean square error of decentralized consensus protocols that preserve the  \startmodif expected average\stopmodif. Our results ensure that, under mild conditions, distributed averaging can be performed via asymmetric and asynchronous algorithms, with a loss in the quality of the estimate which vanishes when the number of samples (and nodes) \startmodif is increased. \stopmodif This fact strongly supports the application of \startmodif these \stopmodif algorithms to large networks. Our analysis \startmodif complements the results about the speed of convergence, which has been \stopmodif thoroughly studied in the literature and was not reconsidered in this paper.
Regarding design issues, we indeed note that optimizing an algorithm for accuracy may entail a slower convergence rate: for instance, it is intuitive that in the AAGA, larger values of $q$ imply faster convergence but poorer accuracy. Thanks to our results, the speed/accuracy trade-off can be more precisely studied in a wide range of examples.

Unlike certain previous approaches, which relied on the convergence speed of these systems, our results are based on the fact that the increase of the error can be bounded proportionally to the decrease of the disagreement. 
As such, they are independent of the speed at which the system converges, and therefore of the spectral properties of the network, which determines this speed. 
Notably, our bounds only involve local quantities such as the degree of the nodes or the weight that they give to their neighbors' values, as opposed to global ones such as the eigenvalues of the network Laplacian. 
As local quantities are much easier to control in distributed systems, our results are of immediate application in design.

  Our method has been applied to several known protocols: although we have sometimes been very conservative when deriving our bounds, we have obtained bounds that either match or improve upon those available in the literature. In addition, results from algorithmic simulations are closely matched by our bounds, which appear to accurately capture the qualitative dependence on  the network size. Note that we have limited the number of applications of our results presented here, in the interest of concision and simplicity: some additional applications can be found in~\cite{PF-JMH:13-ecc}. 

Overall, two classes of systems were proved to be {\em accurate}: those with sufficiently few or small simultaneous updates, and those with sufficiently uncorrelated simultaneous updates. These two apparently   unrelated situations in reality present strong similarities, because the updates taking place at different times are assumed to be uncorrelated. \stopmodif This suggests that the real parameter,   which determines \stopmodif the mean square error, is the level of correlation between the updates taking place across the history of the system. Further work could be devoted   to formalize and quantify this intuition on the importance of the correlations between the updates.
Finally, we note that --to the best of our knowledge-- the distribution of the final values for processes which do not preserve the expected average has not been studied yet.   \stopmodif

\bibliographystyle{plain}

\end{document}